\documentclass[12pt]{amsart}
\usepackage{amsmath, amsthm, amscd, amsfonts}

\newtheorem{theorem}{Theorem}[section]
\newtheorem{lemma}[theorem]{Lemma}
\newtheorem{proposition}[theorem]{Proposition}
\newtheorem{corollary}[theorem]{Corollary}
\theoremstyle{definition}
\newtheorem{definition}[theorem]{Definition}
\newtheorem{example}[theorem]{Example}

\theoremstyle{remark}

\numberwithin{equation}{section}

\newfont{\kh}{msbm10}

\begin{document}
\title[EP modular operators and their products]
{EP modular operators and their products}
\author{K. Sharifi}
\address{Kamran Sharifi, \newline Permanent Address:
Department of Mathematics,
Shahrood University of Technology, P. O. Box 3619995161-316,
Shahrood, Iran
\newline
Current Address: Mathematisches Institut,
Universit\"{a}t M\"{u}nster,
Einsteinstrasse 62, 48149 M\"{u}nster, Germany}
\email{sharifi.kamran@gmail.com and
sharifi@shahroodut.ac.ir}


\begin{abstract}
We study first EP modular operators on Hilbert C*-modules and
then we provide necessary and sufficient conditions for the
product of two EP modular operators to be EP. These enable us to
extend some results of Koliha  [{\it Studia Math.} {\bf 139}
(2000), 81--90.] for an arbitrary C*-algebra and the C*-algebras
of compact operators.

\end{abstract}
\maketitle

\section{Introduction.}

A bounded linear operator $T$ with closed range on a complex
Hilbert space $H$ is called an EP operator if $T$ and $T^*$ have
the same range. This was introduced for matrices by Schwerdtfeger
in \cite{Schwerdtfeger} and has been studied in detail by several
authors, see e.g. \cite{Boasso1, Boasso2, DKS, D-K1,
KolihaCommuting, lesnjak, MDK} and references therein. A problem that has
been open for over twenty-five years is when the product of two
EP matrices is again EP \cite{BaskettKatz}. Hartwig and Katz
\cite{HartwigKatz}, and Koliha \cite{KolihaSimpleProof} gave
necessary and sufficient conditions for a product of two $n
\times n$ complex EP matrices to be EP. Djordjevi\'c
\cite{Djordjevic} provided a generalization of the result for EP
operators on Hilbert spaces. In this note we investigate about
the EP operators on Hilbert C*-modules over an arbitrary
C*-algebra of coefficients, and then we reformulate some results
of \cite{KolihaSimpleProof, KolihaCommuting} for the product of
EP modular operators.

Since the finite-dimensional spaces, Hilbert spaces and
C*-algebras can all be regarded as Hilbert C*-modules, one can
study EP modular operators in a unified way in the framework of
Hilbert C*-modules. Indeed, a Hilbert C*-module is an object
like a Hilbert space except that
the inner product is not scalar-valued, but takes its values in a
C*-algebra of coefficients. Since the geometry of these modules
emerges from the C*-valued inner product, some basic properties
of Hilbert spaces like Pythagoras' equality, self-duality, and
decomposition into orthogonal complements must be given up. These
modules play an important role in the modern theory of
C*-algebras and the study of locally compact quantum groups. A
(right) {\it pre-Hilbert C*-module} over a C*-algebra
$\mathcal{A}$ is a right $\mathcal{A}$-module $X$ endowed with an
$\mathcal{A}$-valued inner product $\langle \cdot , \cdot \rangle
: X \times X \to \mathcal{A}, \ (x,y) \mapsto \langle x,y
\rangle$ which is linear in the second variable $y$ (and
conjugate-linear in $x$), satisfying the conditions
$$ \langle x,y \rangle=\langle y,x \rangle ^{*}, \ \langle x,ya \rangle
=\langle x,y \rangle a \ {\rm for} \ {\rm all} \ a \in
\mathcal{A},$$
$$ \langle x,x \rangle \geq 0 \ \ {\rm with} \
   {\rm equality} \ {\rm if} \ {\rm and} \ {\rm only} \
   {\rm if} \ x=0.$$

A pre-Hilbert $\mathcal{A}$-module $X$ is called a {\it Hilbert $
\mathcal{A}$-module} if $X$ is a Banach space with respect to the
norm $\| x \|=\|\langle x,x\rangle \| ^{1/2}$. If $X$, $Y$ are two
Hilbert $ \mathcal{A}$-modules then the set of all ordered pairs
of elements $X \oplus Y$ from $X$ and $Y$ is a Hilbert
$\mathcal{A}$-module with respect to the $\mathcal A$-valued
inner product $\langle (x_{1},y_{1}),(x_{2},y_{2})\rangle=
\langle x_{1},x_{2}\rangle_{X}+\langle y_{1},y_{2}\rangle _{Y}$.
It is called the direct {\it orthogonal sum of $X$ and $Y$}. If
$V$ is a (possibly non-closed) $\mathcal A$-submodule of $X$, then
$V^\bot :=\{ y \in X: ~ \langle x,y \rangle=0~ \ {\rm for} \ {\rm
all}\ x \in V \} $ is a closed $\mathcal A$-submodule of $X$ and
$ \overline{V} \subseteq V^{ \perp \, \perp}$. A Hilbert $\mathcal
A$-submodule $V$ of a Hilbert $\mathcal A$-module $X$ is
orthogonally complemented if $V$ and its orthogonal complement
$V^\bot$ yield $X=V \oplus V^\bot $, in this case, $V$ and its
biorthogonal complement $V^{ \perp \, \perp}$ coincide. For the
basic theory of Hilbert C*-modules we refer to the books
\cite{LAN, M-T}.

Throughout the present paper we assume $\mathcal{A}$ to be an
arbitrary C*-algebra (i.e. not necessarily unital). We use the
notations $Ker(\cdot)$ and $Ran(\cdot)$ for kernel and range of
operators, respectively. We denote by $\mathcal{L}(X,Y)$ the
Banach  space of all bounded adjointable operators between $X$ and
$Y$, i.e., all bounded $\mathcal A$-linear maps $T :X \rightarrow
Y$ such that there exists $T^*:Y \rightarrow X$ with the property
$ \langle Tx,y \rangle =\langle x,T^*y \rangle$ for all $ x \in
X$, $y \in Y$. The C*-algebra $\mathcal{L}(X,X)$ is abbreviated
by $\mathcal{L}(X)$.

In this paper we first briefly investigate some basic facts about
EP modular operators with closed ranges and then we give some factorizations
and characterizations of
such operators. If $T, S$ and $TS$ are EP
modular operators with closed ranges then $Ran(TS)=Ran(T)\cap Ran(S)$.
If, in addition, $Ker(T) + Ker(S)$ is dense in its biorthogonal complement
then we obtain $Ker(TS)= \overline{ Ker(T) + Ker(S)}$. Some special cases for
EP elements of C*-algebras and C*-algebras of compact operators are considered.

\section{Preliminaries}
Closed submodules of Hilbert modules need not to be orthogonally
complemented at all, but Lance states in \cite[Theorem 3.2]{LAN}
under which conditions closed submodules may be orthogonally
complemented. Let $X$ be a Hilbert $\mathcal{A}$-module and
suppose that an operator $T$ in $ \mathcal{L}(X)$ has closed
range, then one has:

\begin{itemize}
\item $Ker(T)$ is orthogonally complemented in $X$, with
      complement $Ran(T^*)$,
\item $Ran(T)$ is orthogonally complemented in $X$, with
      complement $Ker(T^*)$,
\item the map $T^* \in \mathcal{L}(X)$ has closed range, too.
\end{itemize}

The following results express when the product of two modular
operators with closed range again has closed range. Suppose $T, S \in \mathcal{L}(X)$ are
bounded adjointable operators with closed range. Then $TS$ has closed range,
 if and only if $Ker(T)+Ran(S)$ is an orthogonal summand in $X$ if an only if
 $Ker(S^*)+Ran(T^*)$ is an orthogonal summand in $X$. For the proof
of the results and historical notes about the problem we refer to
\cite{SHA/PRODUCT} and references therein.

\newcounter{cou001}
Let $T \in \mathcal{L}(X)$, then a bounded adjointable operator
$T^{ \dag} \in \mathcal{L}(X)$ is called the {\it Moore-Penrose
inverse} of $T$ if
\begin{equation} \label{MPinverse}
T \, T^{ \dag}T=T, \ T^{ \dag}T \, T^{ \dag}= T^{ \dag}, \ (T \,
T^{ \dag})^*=T \, T^{ \dag} \ {\rm and} \ ( T^{ \dag} T)^*= T^{
\dag} T.
\end{equation}
The notation $T^{ \dag}$ is reserved to denote the Moore-Penrose
inverse of $T$. These properties imply that $T^{ \dag}$ is unique
and $ T^{ \dag} T$ and $ T \, T^{ \dag} $ are orthogonal
projections. Moreover, $Ran( T^{ \dag} )=Ran( T^{ \dag}  T)$,
$Ran( T )=Ran( T \, T^{ \dag})$,  $Ker(T)=Ker( T^{ \dag} T)$ and
$Ker(T^{ \dag})=Ker( T \, T^{ \dag} )$ which lead us to $ X= Ker(
T^{ \dag} T) \oplus Ran( T^{ \dag} T)= Ker(T) \oplus Ran( T^{
\dag} )$ and $X= Ker(T^{ \dag}) \oplus Ran(T).$ If $T^{ \dag}$
exists then $T^{ \dag}= \lim_{ \omega \to 0^+} ( \omega 1+
T^*T)^{-1} T^* = \lim_{ \omega \to 0^+} T( \omega 1+ T^*T)^{-1}$,
cf. \cite{SHA/PARTIAL, SHA/Groetsch}.

 Xu and Sheng in \cite{Xu/Sheng} have shown that a bounded
adjointable operator between two Hilbert C*-modules admits a
bounded Moore-Penrose inverse if and only if the operator has
closed range. The reader should be aware of the fact that a
bounded adjointable operator may admit an unbounded operator as
its Moore-Penrose, see \cite{FS2, SHA/PARTIAL, SHA/Groetsch} for
more detailed information.

\begin{definition}Let $X$ be a Hilbert
$ \mathcal{A}$-modules. An operator $T \in \mathcal{L}(X)$ is
called $EP$ if $Ran(T)$ and $Ran(T^*)$ have the same closure.
\end{definition}

In the Hilbert C*-module context, one needs to add the extra
condition, closedness of the range, in order to get a reasonably
good theory. This ensure that an EP operator has a bounded
adjointable Moore-Penrose inverse. Like the general theory of
Hilbert spaces one can easily see that the following conditions
are equivalent:
\begin{itemize}
\item $T$ is EP with closed range,
\item $T$ and $T^*$ have the same kernel,
\item $T$ is Moore-Penrose invertible and $T \, T^{ \dag}= T^{ \dag} \, T$,
\item $Ran(T)$ is orthogonally complemented in $X$, with
      complement $Ker(T)$.
\end{itemize}

\begin{proposition} \label{EP0} Let $X$ be a Hilbert
$ \mathcal{A}$-module and $T \in \mathcal{L}(X)$ have a closed
range. Then the following conditions are equivalent:
\begin{list}{(\roman{cou001})}{\usecounter{cou001}}
\item $T$ is EP with closed range,
\item there exists an isomorphism $V \in \mathcal{L}(X)$ such that $T^*=VT$,
\item there exists an isomorphism $V \in \mathcal{L}(X)$ such that $T^{
\dag}=VT=TV$.
\end{list}
\end{proposition}

\begin{proof} Suppose $T$ is EP. Then $Ker(T)$ is orthogonally
complemented and there exists $c > 0$ such that $ \| T \, x \|
\geq c \| x \| $ for all $x \in Ker(T)^{\perp}$, cf.
\cite[Proposition 1.3]{F-S}. The latter inequality implies that
the module map $ T_{|_{Ker(T)^{ \perp}}}: Ran(T^*)=Ran(T) \to
Ran(T)$ has a bounded inverse, which allows us to define $
\mathcal{A}$-module map

\[ \qquad V x \, =~
\begin{cases}
T^* \,(T_{| Ker(T)^{ \perp}})^{-1}x ~ & \text{ if   $x \in Ran(T)$}\\
x ~ & \text{ if $x \in Ker(T). $}\
\end{cases}
\]

Then $V \in \mathcal{L}(X)$ is an isomorphism which satisfies
$T^*=VT$, that is, (ii) holds. To prove (iii) we define

\[ \qquad V x \, =~
\begin{cases}
(T_{| Ker(T)^{ \perp}})^{-2}x ~ & \text{ if   $x \in Ran(T)$}\\
x ~ & \text{ if $x \in Ker(T), $}\
\end{cases}
\]
and
\[ \qquad Sx \, =~
\begin{cases}
(T_{| Ker(T)^{ \perp}})^{-1}x ~ & \text{ if   $x \in Ran(T)$}\\
0 ~ & \text{ if $x \in Ker(T). $}\
\end{cases}
\] \\
Then $V, S \in \mathcal{L}(X)$. Using the orthogonal direct sum
decompositions, we have $TV=VT=S$. Moreover, $T$ and $S$ satisfy
equations (\ref{MPinverse}), i.e., $S$ is the Moore-Penrose
inverse of $T$. The remaining parts follow trivially.
\end{proof}

\section{On the product of ep modular operators}
In this section we try to generalize some results of Koliha
\cite{KolihaSimpleProof, KolihaCommuting} to the framework of
Hilbert C*-modules. Some special cases for EP elements of
C*-algebras and the C*-algebra of compact operators are also
obtained.

\begin{lemma} \label{EP71} Suppose $X$ is a Hilbert
$ \mathcal{A}$-module. Let $T \in \mathcal{L}(X)$ have
closed range and $S \in \mathcal{L}(X)$ be
an arbitrary operator which commutes with $T$.  If $T$ is normal or EP, then $S$
commutes with $T^{ \dag}$.
\end{lemma}
\begin{proof} Using Fuglede-Putnam Theorem
\cite[Theorem 2.8]{SHA/NORMALITY}, the operator $S$ commutes with $T^*$.
Hence, $S$ commutes with $( \omega 1+
T^*T)^{-1} T^* $, $\omega > 0$, and its limit $T^{ \dag}$.
\end{proof}

\begin{proposition} \label{EP72} Suppose $X$ is a Hilbert
$ \mathcal{A}$-module. Let $T, S \in \mathcal{L}(X)$ be EP
operators with closed range and $TS=ST$. Then $TS$ is an EP
operator with closed range.
\end{proposition}
\begin{proof} The operators $T, S, T^{ \dag}$ and $S^{ \dag}$ are mutually commute
by Lemma \ref{EP71}. Therefore $TS$ and $S^{ \dag} T^{ \dag}$ satisfies equations
(\ref{MPinverse}), i.e., $TS$ is Moore-Penrose invertible and
its Moore-Penrose inverse equals $S^{ \dag} T^{ \dag}=T^{ \dag}S^{ \dag}$.
These implies that $TS$ has a closed range and $(TS)^{ \dag}$ commutes with
$TS$.
\end{proof}


\begin{proposition} \label{EP06} Let $X$ be a Hilbert
$ \mathcal{A}$-module. If $T, S, TS  \in \mathcal{L}(X)$ are EP
operators then $T(\overline{Ran(S)}) \subseteq \overline{Ran(S)}$
and $S^*(\overline{Ran(T)}) \subseteq \overline{Ran(T)}$.
\end{proposition}
\begin{proof} Continuity of $T$ implies that $T(\overline{Ran(S)})
\subseteq \overline{T(Ran(S))}$ and so $ \overline{
T(\overline{Ran(S)})} = \overline{T(Ran(S))}$. We therefore have
$$T(\overline{Ran(S)}) \subseteq  \overline{ T(\overline{Ran(S)})}
 = \overline{T(Ran(S))}= \overline{Ran(TS)}= \overline{Ran(S^*T^*)}
\subseteq \overline{Ran(S)}.$$ The second inclusion follows in a
similar manner.
\end{proof}

Koliha in \cite{KolihaSimpleProof, KolihaCommuting} demonstrated
that the converse of the above statement is true in the case of
finite dimensional spaces. Indeed, he showed that for matrices $T$
and $S$, $TS$ is EP if and only if $Ker(T) \subseteq Ker(TS)$ and
$Ran(TS) \subseteq Ran(S)$. However, the statement does not hold in
the case of Hilbert space or Hilbert C*-modules.

\begin{example}Let $\mathcal{A}$ be unital
C*-algebra and $H_{ \mathcal{A}}$ be the standard Hilbert $
\mathcal{A}$-module which is countably generated by orthonormal
basis $ \xi _{j}=(0,...,0,1,0,...,0),~ j \in \mathbb{N}$. Let
$W=\overline{span \{\xi _{2j}:~ j \in \mathbb{N}\}}$ and $S$ be
the orthogonal projection onto the closed submodule $W$. We
define $T_0$ by $T_0( \xi_1)= \xi_2$, $T_0( \xi_{2j})=\xi_{2j+2}$
and $ T_0( \xi_{2j+1})=\xi_{2j-1}$, for all $j \in \mathbb{N}$.
The inverse of $T_0$ is defined by $T_0^{-1} (\xi_2)= \xi_1$,
$T_0^{-1}( \xi_{2j})=\xi_{2j-2}$
and $ T_0^{-1}(\xi_{2j+1})= \xi_{2j+3}$. Then $T_0$ and $ T_0^{-1}$
can be extended uniquely to $T$ and $T^{-1}$ on $H_{ \mathcal{A}}$
which satisfy $T^*=T^{-1}$. One can easily see that
$T(\overline{Ran(S)}) \subseteq \overline{Ran(S)}$ and
$S^*(\overline{Ran(T)})= Ran(S^*)=Ran(S)=W \subseteq \overline{Ran(T)}$.
However, $TS$ is not EP since $ \xi_2 $ is orthogonal to $ \overline{Ran(TS)}$
and to $Ker(TS)$.
\end{example}

\begin{lemma} \label{EP1} Let $X$ and $Y$ be a Hilbert
$ \mathcal{A}$-modules and $T \in \mathcal{L}(X,Y)$ have a closed
range. If $ \mathcal{A}$-submodule $W$ is orthogonally
complemented in $Y$ then the operator $T$ has a matrix
representation with respect to the orthogonal sums $X=Ran(T^*)
\oplus Ker(T)$ and $Y=W \oplus W^{ \perp}$ as follows:

\begin{equation} \label{EP2}
T=\begin{bmatrix} T_1 & 0 \\ T_2 & 0 \end{bmatrix}:
\begin{bmatrix} Ran(T^*) \\ Ker(T) \end{bmatrix} \to
\begin{bmatrix} W  \\ W^{ \perp}  \end{bmatrix}.
\end{equation}
In this case, $A=T_{1}^{*} \, T_1 + T_2 \, T_{2}^{*} : Ran(T^*)
\to Ran(T^*)$ is invertible. Moreover,

\begin{equation} \label{EP3}
T^{ \dag}=\begin{bmatrix} A^{-1}T_{1}^{*} & A^{-1}T_{2}^{*} \\ 0 &
0
\end{bmatrix}.
\end{equation}
\end{lemma}

\begin{proof}The operator $T$ in $\mathcal{L}(Ran(T^*) \oplus Ker(T), W \oplus W^{ \perp})$
can be represented by the matrix

\begin{equation} \label{EP22}
T=\begin{bmatrix} T_1 & T_3 \\ T_2 & T_4 \end{bmatrix}:
\begin{bmatrix} Ran(T^*) \\ Ker(T) \end{bmatrix} \to
\begin{bmatrix} W  \\ W^{ \perp}  \end{bmatrix},
\end{equation}

in which,

\begin{eqnarray*}
T_1 &=& T_{|_{Ran(T^*)}} : Ran(T^*) \to W, \\
T_2 &=& T_{|_{Ran(T^*)}} : Ran(T^*) \to W^{ \perp}, \\
T_3 &=& T_{|_{Ker(T)}}  ~ : Ker(T) \to W, \\
T_4 &=& T_{|_{Ker(T)}}  ~ : Ker(T) \to W^{ \perp}.
\end{eqnarray*}
Since $T( Ker(T))=0$, we obtain $T_3=T_4=0$. Suppose $x \in
Ran(T^*)$ and $Ax=0$ then $0= \langle T_{1}^{*} \, T_1 + T_2 \,
T_{2}^{*}x,x \rangle= \langle T_1x,T_1x \rangle + \langle
T_2x,T_2x \rangle \geq \langle T_1x,T_1x \rangle \geq 0$, which
implies $T_1x=0$. Similarly, $T_2x=0$ and so $x \in Ker(T) \cap
Ran(T^*)=\{ 0 \}$, i.e., $A$ is injective. Using closedness of the
range of $T^*$ and \cite[Lemma 2.1]{SHA/PRODUCT}, we have
$Ran(T^* T)=Ran(T^*)$ which follows subjectivity of $A$. Hence,
$A$ is invertible. Finally, the operators $T= \left[
\begin{smallmatrix} T_1&0\\ T_2&0 \end{smallmatrix} \right]$ and
 $ \left[ \begin{smallmatrix} A^{-1}T_{1}^{*} & A^{-1}T_{2}^{*} \\ 0 &
0 \end{smallmatrix} \right]$ satisfy equations (\ref{MPinverse}) which obtain the
matrix form for the Moore-Penrose inverse of $T$.
\end{proof}

\begin{lemma} \label{EP4} Let $X$ be a Hilbert
$ \mathcal{A}$-module and $T \in \mathcal{L}(X)$ with closed
range. Then $T$ is EP if and only if it is of the matrix form

\begin{equation} \label{EP5}
T=\begin{bmatrix} T_1 & 0 \\ 0 & 0 \end{bmatrix}:
\begin{bmatrix} Ran(T) \\ Ker(T) \end{bmatrix} \to
\begin{bmatrix} Ran(T)  \\ Ker(T)  \end{bmatrix},
\end{equation}
for some invertible operator $T_1 \in \mathcal{L}(Ran(T),Ran(T))$.
\end{lemma}

\begin{proof}Let $T$ be EP then the bounded adjointable
$T_1: Ker(T)^{ \perp}=Ran(T) \to Ran(T)$, $T_1x=Tx$, has a
bounded inverse. This fact together with the matrix representation
(\ref{EP2}) give us the desired representation.

Conversely, let $T_1 \in \mathcal{L}(Ran(T),Ran(T))$ and $T$ admit
the matrix representation (\ref{EP5}). Then
\begin{equation} \label{EP33}
T^{ \dag}= \begin{bmatrix} (T_1^{*} \,  T_1)^{-1}T_{1}^{*} & 0 \\
0 & 0 \end{bmatrix}= \begin{bmatrix} T_1^{ \, -1} & 0 \\
0 & 0 \end{bmatrix}
\end{equation}
is the Moore-Penrose inverse of $T$ which commutes with $T$,
i.e., $T$ is EP.
\end{proof}

Suppose $M$ and $N$ are submodule of a Hilbert C*-module $X$, then
$(M+N)^{ \perp}=M^{ \perp} \cap N^{ \perp}$. In particular, if $
M+N$ is dense in its biorthogonal complement then
$$(M^{ \perp} \cap N^{ \perp})^{ \perp} = (M+N)^{ \perp \, \perp}= \overline{M+N}.$$

\begin{theorem} \label{EP6} Let $X$ be a Hilbert
$ \mathcal{A}$-module and $T, S \in \mathcal{L}(X)$ are EP
operators with closed ranges. Among the following four properties
of $T$, $S$ and $TS$, the implication $(i) \rightarrow  (iii)$
holds. Moreover, (i) and (ii) are equivalent to (iii) and (iv).
\begin{list}{(\roman{cou001})}{\usecounter{cou001}}
\item $TS$ is an EP operator with closed range.
\item $Ker(T) + Ker(S)$ is dense in its biorthogonal complement.
\item $Ran(TS)=Ran(T)\cap Ran(S)$.
\item $Ker(TS)= \overline{ Ker(T) + Ker(S)}$.
\end{list}
\end{theorem}

\begin{proof}Suppose $T$, $S$ and $TS$ are EP operators with closed ranges.
Using Lemmata \ref{EP1}, \ref{EP4} and the orthogonal sums $X=Ker(T) \oplus Ran(T)$
and $X=Ker(S) \oplus Ran(S)$, we get the matrix decompositions

\begin{equation} \label{EP66}
T=\begin{bmatrix} T_1 & 0 \\ 0 & 0 \end{bmatrix}:
\begin{bmatrix} Ran(T) \\ Ker(T) \end{bmatrix} \to
\begin{bmatrix} Ran(T)  \\ Ker(T)  \end{bmatrix}
\end{equation}
and
\begin{equation} \label{EP666}
S=\begin{bmatrix} S_1 & 0 \\ S_2 & 0 \end{bmatrix}:
\begin{bmatrix} Ran(S) \\ Ker(S) \end{bmatrix} \to
\begin{bmatrix} Ran(T)  \\ Ker(T)  \end{bmatrix},
\end{equation}
with respect to the orthogonal sums of submodules. Moreover, the
adjoint of $S$ is given by
\begin{equation} \label{EP6666}
S^*=\begin{bmatrix} S_1^* & S_2^* \\
0 & 0 \end{bmatrix}: \begin{bmatrix} Ran(T)  \\ Ker(T)
\end{bmatrix} \to \begin{bmatrix} Ran(S) \\ Ker(S)
\end{bmatrix} .
\end{equation}
Since $T_1 : Ran(T) \to Ran(T)$ is invertible we obtain $ Ker(TS)
\cap Ran(S)= Ker(S_1),$ which implies that
\begin{equation} \label{EP61}
Ker(TS)=Ker(TS) \cap ( Ker(S) \oplus Ran(S))=Ker(S) \oplus Ker(S_1).
\end{equation}
Since $TS$ is EP, $Ran(TS)=Ran(S^*T^*) \subseteq Ran(S^*)=Ran(S) $,
which implies $Ran(TS) \subseteq Ran(T) \cap Ran(S)$.

For the converse, let $y \in Ran(T) \cap Ran(S)$ and $y \in
Ran(TS)^{ \perp}=Ker(TS)=Ker(S) \oplus Ker(S_1)$. Then $S_1 y
=0$. Using the matrix decompositions (\ref{EP666}) and
(\ref{EP6666}) for $S$ and $S^*$, respectively, and fact that
$Ran(S)=Ran(S^*)$, we obtain
\begin{equation} \label{EP64}
\begin{bmatrix} S_1 & 0 \\
S_2 & 0 \end{bmatrix}
\begin{bmatrix} y \\ 0 \end{bmatrix}=
\begin{bmatrix} S_1^* & S_2^* \\
0 & 0 \end{bmatrix}
\begin{bmatrix} z \\ w \end{bmatrix},
\end{equation}
for some $z \in Ran(T)$ and $w \in Ket(T)$. Then $S_1y=S_2 y=0$,
for every $y \in Ran(T) \cap Ran(S)$, that is,
$$ \begin{bmatrix} S_{1} & 0 \\ S_{2} & 0 \end{bmatrix}
\begin{bmatrix} y \\ 0 \end{bmatrix}= \begin{bmatrix} 0 \\ 0 \end{bmatrix}.$$
Hence, $y \in Ker(S) \cap Ran(S)=\{ 0 \}$ which yields $Ran(T)
\cap Ran(S) \subseteq Ran(TS)$. That is, the implication $(i)
\rightarrow  (iii)$ holds.

Suppose $TS$ is an EP operator with closed range and $Ker(T) +
Ker(S)$ is dense in its biorthogonal complement. Then $
Ran(TS)=Ran(T)\cap Ran(S)$ by the preceding argument, which
implies that
\begin{eqnarray*}
\overline{ Ker(T) + Ker(S)}= (Ker(T) + Ker(S))^{ \perp \,
\perp}&=& ( Ran(T)\cap Ran(S))^{ \perp} \\
      &=&  Ran(TS)^{ \perp}=Ker(TS).
\end{eqnarray*}

Conversely, suppose $Ran(TS)=Ran(T)\cap Ran(S)$ and $Ker(TS)=
\overline{ Ker(T) + Ker(S)}$ then $Ran(TS)$ is a closed submodule
as an intersection of two closed submodules $Ran(T)$ and $Ran(S)$.
Hence, $Ran((TS)^*)$ is orthogonally complemented in $X$, cf.
\cite[Theorem 3.2]{LAN}. We find
\begin{eqnarray*}
Ran((TS)^*)=Ker(TS)^{ \, \perp} &=& \overline{ Ker(T) + Ker(S)}^{
\, \perp} \\
&=& Ran(T)\cap Ran(S)= Ran(TS),
\end{eqnarray*}
i.e. $TS$ is an EP operator with closed range.
\end{proof}

Recall that a C*-algebra of compact operators is a $c_{0}$-direct
sum of elementary C*-algebras $\mathcal{K}(H_{i})$ of all compact
operators acting on Hilbert spaces $H_{i}, \ i \in I$, i.e.
$\mathcal{A} = c_{0}$-$ \oplus_{i \in I}\mathcal{K} (H_{i})$,
cf.~\cite[Theorem 1.4.5]{ARV}. Suppose $ \mathcal{A}$ is an
arbitrary C*-algebra of compact operators. It is well known that
every norm closed submodule of every Hilbert $
\mathcal{A}$-module is automatically an orthogonal summand.
Further generic properties of the category of Hilbert C*-modules
over C*-algebras which characterize precisely the C*-algebras of
compact operators have been found in \cite{FR1, F-S, FS2} and
references therein. We can reformulate Theorem \ref{EP6} in
terms of bounded $\mathcal{A}$-linear maps on Hilbert C*-modules
over C*-algebras of compact operators.

\begin{corollary}\label{EPA1}Suppose $ \mathcal{A}$ is an
arbitrary C*-algebra of compact operators, $X$ is a Hilbert $
\mathcal{A}$-module and $T, S \in \mathcal{L}(X)$ are EP
operators with closed range. Then $TS$ is an EP operator with
closed range if and only if $Ran(TS)=Ran(T)\cap Ran(S)$ and
$Ker(TS)= \overline{ Ker(T) + Ker(S)}$.
\end{corollary}

Koliha in \cite{KolihaCommuting} gave necessary and sufficient
conditions for elements of C*-algebras which commute with their
Moore-Penrose inverse. He also studied conditions which ensure
that the property is preserved under multiplication. As a special
case of our results we recover some parts of Theorem 4.3 of
\cite{KolihaCommuting}.
\begin{corollary}Suppose $ \mathcal{A}$ is an
arbitrary C*-algebra and $a, \, b$ and $ab$ commute with their
Moore-Penrose inverse. Then $ab  \mathcal{A}= a  \mathcal{A} \cap b
\mathcal{A}$.
\end{corollary}

\begin{corollary}\label{EPA2}Suppose $ \mathcal{A}$ is an
arbitrary C*-algebra of compact operators and $a$ and $b$
commute with their Moore-Penrose inverse. Then $ab$
commutes with its Moore-Penrose inverse if and only if
 $ab  \mathcal{A}= a  \mathcal{A} \cap b
\mathcal{A}$
 and $(ab)_{-1}= \overline{ a_{-1}(0)+ a_{-1}(0)}$, in which
 $ a_{-1}(0)=\{ x \in \mathcal{A}:~xa=0 \}$.
\end{corollary}

Recall that every C*-algebra is an $ \mathcal{A}$-module on its
own and define the bounded operators $L_a:\mathcal{A} \to
\mathcal{A}$, $L_a(x)=ax$ then $a^{-1}(0)= Ker(L_a)$ and
$Ran(L_a)=a \, \mathcal{A}$. The above facts follows from Theorem
\ref{EP6} and Corollary \ref{EPA1}.

{\bf Acknowledgement}: The author would like to thank
the referee for his/her careful reading and useful comments.

\end{document}